\documentclass[a4paper]{amsart}

\usepackage[utf8]{inputenc}
\usepackage{mathtools}
\usepackage{amssymb,amsfonts,amsmath}
\usepackage{enumerate}
\usepackage{color}
\usepackage{mathrsfs}
\usepackage{graphicx}
\usepackage{verbatim}
\usepackage{tikz-cd}

\newtheorem{theorem}{Theorem}[section]
\newtheorem{corollary}[theorem]{Corollary}
\newtheorem{lemma}[theorem]{Lemma}
\newtheorem{proposition}[theorem]{Proposition}
\newtheorem{question}[theorem]{Question}

\theoremstyle{definition}
\newtheorem{definition}[theorem]{Definition}
\newtheorem{notation}[theorem]{Notation}

\theoremstyle{remark}
\newtheorem{remark}[theorem]{Remark}

\newtheorem{example}[theorem]{Example}

\newcommand{\F}{\mathbb{F}}
\newcommand{\N}{\mathbb{N}}

\newcommand{\R}{\mathbb{R}}
\newcommand{\Z}{\mathbb{Z}}

\DeclareMathOperator{\Aut}{Aut}
\DeclareMathOperator{\CSP}{CSP}

\DeclareMathOperator{\Fix}{Fix}

\DeclareMathOperator{\id}{id}

\DeclareMathOperator{\lv}{lv}

\DeclareMathOperator{\res}{res}

\DeclareMathOperator{\RiSt}{RiSt}

\DeclareMathOperator{\SL}{SL}
\DeclareMathOperator{\SAff}{SAff}

\DeclareMathOperator{\St}{St}
\DeclareMathOperator{\supp}{supp}
\DeclareMathOperator{\Sym}{Sym}
\DeclareMathOperator{\vol}{vol}
\newcommand{\abs}[1]{\vert #1 \vert}
\newcommand\Set[2]{\{\,#1\mid#2\,\}}

\newcommand{\defeq}{\mathrel{\mathop{:}}=}

\renewcommand{\epsilon}{\varepsilon}

\title[Amenability and profinite completions]{Amenability and profinite completions of finitely generated groups}
\author[S. Kionke]{Steffen Kionke}
\author[E. Schesler]{Eduard Schesler}
\address{FernUniversit\"at in Hagen \\ Fakult\"at f\"ur Mathematik und Informatik \\
58084 Hagen}
\email{steffen.kionke@fernuni-hagen.de}
\email{eduard.schesler@fernuni-hagen.de}
\thanks{Funded by the Deutsche Forschungsgemeinschaft (DFG, German Research Foundation) - 441848266}

\begin{document}
\begin{abstract}
This article explores the interplay between the finite quotients of finitely generated residually finite groups and the concept of amenability.
We construct a finitely generated, residually finite, amenable group $A$ and an uncountable family of finitely generated, residually finite non-amenable groups all of which are profinitely isomorphic to $A$. All of these groups are branch groups. 
Moreover, picking up  Grothendieck's problem, the group $A$ embeds in these groups such that the inclusion induces an isomorphism of profinite completions.  

In addition, we review the concept of uniform amenability, a strengthening of amenability introduced in the 70's, and we prove that uniform amenability indeed is detectable from the profinite completion.
\end{abstract}
\maketitle

\section{Introduction}
The profinite completion $\widehat{G}$ of a group $G$ is the inverse limit of its finite quotients. If $G$ is residually finite, then $G$ embeds into $\widehat{G}$ and it is natural to wonder what properties of $G$ can be detected from $\widehat{G}$. Especially a question of Grothendieck, posed in 1970 \cite{Grothendieck70}, sparked interest in this direction:  Is an embedding $\iota \colon H \to G$ of finitely presented, residually finite groups an isomorphism, if it induces an isomorphism $\widehat{\iota}\colon \widehat{H} \to \widehat{G}$ of profinite completions?
The answer is negative. Finitely generated counterexamples were constructed already in 1986 by Platonov and Tavgen \cite{PlatonovTavgen}. The finitely presented case was settled almost 20 years later by Bridson and Grunewald~\cite{BridsonGrunewald}.

In this article we explore the interplay between amenability and the profinite completion of finitely generated groups. Our interest was prompted by the following variation of Grothendieck's problem:

\smallskip

$(*)$ \textit{Given two finitely generated residually finite groups $A$, $G$, where $A$ is amenable. Suppose $\iota\colon A \to G$ induces an isomorphism $\widehat{\iota}\colon \widehat{A} \to \widehat{G}$. Is $G$ amenable?}

\smallskip

The answer is negative and we obtain the following result (a consequence of Theorem \ref{thm:main-theorem-precise} below). 
\begin{theorem}\label{thm:main-theorem}
There is an uncountable family of pairwise non-isomorphic, residually finite $18$-generator groups $(G_j)_{j\in J}$ and a residually finite $6$-generator group $A$ with embeddings $\iota_j \colon A \to G_j$ such that the following properties
 hold:
\begin{enumerate}
\item $\widehat{\iota_j} \colon \widehat{A} \rightarrow \widehat{G_j}$ is an isomorphism,
\item $A$ is amenable,
\item each $G_j$ contains a non-abelian free subgroup.
\end{enumerate}
In particular, amenability is not a profinite invariant of finitely generated residually finite groups.
\end{theorem}
We note that the fibre product construction used in \cite{PlatonovTavgen} and \cite{BridsonGrunewald} is unable to provide such examples, since a fibre product  $P \subseteq H \times H$ projects onto $H$ and thus, it is non-amenable exactly if $H$ is non-amenable (e.g., a non-elementary word hyperbolic group as in \cite{BridsonGrunewald}). Uncountable families of finitely generated \emph{amenable} groups with isomorphic profinite completions  were constructed in \cite{Nekrashevych14,Pyber04}.

 The groups in Theorem \ref{thm:main-theorem} are just-infinite branch groups with the congruence subgroup property. The construction is inspired by a method of Segal \cite{Segal01} and influenced by ideas of Nekrashevych \cite{Nekrashevych14}. The method is rather flexible and allows us to merge a \emph{perfect} residually finite group (e.g.~$\SL_n(\Z)$) with a related amenable group in such a way that the amenable group and the merged group have isomorphic profinite completions. The first step bears similarity with the Sidki-Wilson construction of branch groups with non-abelian free subgroups \cite{SidkiWilson03}.

Let us note that without the requirement of finite generation in $(*)$ there are obvious counterexamples. For instance, $A = \bigoplus_n \mathrm{Alt}(n)$  is amenable as a direct limit of finite groups and $G = \prod_n \mathrm{Alt}(n) = \widehat{A} = \widehat{G}$ contains a non-abelian free group. It would be interesting to have finitely presented counterexamples to $(*)$.
Since our family $(G_j)_{j \in J}$ is uncountable, it is clear that most of the groups cannot be finitely presented. We were unable to verify that none of these groups  admits a finite presentation.

\smallskip

Thinking of amenability as a concept of analytical nature  (e.g., existence of a left invariant mean on $\ell^\infty(G)$), it doesn't seem surprising that the answer to $(*)$ is negative.  From a different perspective, though, amenability is not far from being detectable on finite quotients. 
H.~Kesten \cite{Kesten59} characterized amenability in terms of the spectral radius of symmetric random walks. For a residually finite group $G$ a random walk can surely be studied by looking at large finite quotients of $G$. Trying to exploit this relation one realizes that a \emph{uniform} behavior of all random walks on all finite quotients can be used to deduce amenability of $G$. 

In the 70's G.~Keller defined \cite{Keller72} a notion of \emph{uniform amenability} by imposing that
the size of an  $\varepsilon$-F{\o}lner set for a generating set $S$ can be uniformly bounded in terms of $\epsilon$ and $|S|$ (see Definition \ref{def:uniform-foelner}). A couple of years later the concept was independently defined by Bo\.zejko \cite{Bozejko80}. Wysocza\'nski \cite{Wyso88} showed that
uniform amenability can be characterized in terms of a uniform Kesten condition. This leads to the following result.

 \begin{theorem}\label{thm:uniform-amenability-profinite}
 Let $G_1, G_2$ be residually finite groups with $\widehat{G}_1 \cong \widehat{G}_2$. Then $G_1$ is uniformly amenable if and only if $G_2$ is uniformly amenable.
 \end{theorem}
 As of today the collection of finitely generated groups which are amenable but not uniformly so is rather small and this explains why the construction of counterexamples to $(*)$ actually requires some effort. In turn, the group $A$ of Theorem~\ref{thm:main-theorem} is a new example of an amenable group which is not uniformly amenable.

 There are several ways to prove Theorem \ref{thm:uniform-amenability-profinite} and already Keller's results \cite{Keller72} point in this direction. Here we don't work out the random walk argument sketched above. Instead, we first establish new characterizations of uniform amenability in terms of a uniform isoperimetric inequality and a uniform Reiter condition and use them to give a short proof of the theorem. In addition, we give a new short proof that a uniformly amenable group satisfies a law (a result of Keller \cite{Keller72}). This implies that the profinite completion of a uniformly amenable group is positively finitely generated in the sense of A.~Mann~\cite{Mann96}. Our results on uniform amenability are discussed in Section~\ref{sec:uniform-amenability}.
 
 \medskip
 
We now give an overview of the remaining sections. In Section~\ref{sec:trees-and-construction} we present the basic construction which will be applied throughout. This construction takes a perfect, self-similar subgroup $G \leq \Aut(T_X)$ of the automorphism group of a regular rooted tree $T_X$ and produces -- under a condition introduced by Segal \cite{Segal01} -- a 
branch group $\Gamma_G^\Omega \leq \Aut(T_X)$. The construction depends on a certain subset $\Omega$ of the boundary of the tree.
In Section~\ref{sec:csp} we show that these branch groups are just infinite and have the congruence subgroup property. We deduce that the profinite completion is always an iterated wreath product which only depends on the action of $G$ on the first level of the tree. As a consequence we obtain a zoo of groups with isomorphic completions and inclusions between these groups induce profinite isomorphisms. In Section \ref{sec:uncountable} we use a rigidity result of Lavrenyuk and Nekrashevych \cite{LavrenyukNekrashevych02} to show that the construction (without additional assumptions) gives rise to an uncountable family of pairwise non-isomorphic groups. Finally we discuss concrete examples in Sections \ref{sec:matrix-groups} and \ref{sec:amenable}. To obtain the amenable group $A$ in Theorem \ref{thm:main-theorem} we apply the construction to the special affine group $\F_p^n \rtimes \SL_n(\F_p)$ acting on the $p^n$-regular rooted tree by rooted automorphisms. It follows from a result of Bartholdi, Kaimanovich and Nekrashevych \cite{BartholdiKaimanovichNekrashevych10} that the result is an amenable group (for suitable parameters $\Omega$). On the other hand, we apply the construction to the special affine group $\Z^n \rtimes \SL_n(\Z)$ which acts self-similarly on the $p^n$-regular rooted tree (obtained from the pro-$p$ completion of $\Z^n$). Merging these groups with $A$, we obtain the family $(G_j)$ of non-amenable groups in Theorem~\ref{thm:main-theorem}.

\section{Uniformly amenable classes of groups}\label{sec:uniform-amenability}
In this section we study \emph{uniform amenability} of groups and classes of groups. The concept was introduced by Keller in \cite{Keller72} and independently by Bo\.zejko \cite{Bozejko80}. Here we establish new equivalent characterizations and use them to prove that uniform amenability is a profinite property. Let us begin with the usual characterization using a \emph{uniform F\o{}lner condition}.
\begin{definition}\label{def:uniform-foelner}
A class of groups $\mathfrak{K}$ is \emph{uniformly amenable}
if there is a function $m\colon \R_{>0} \times \N \to \N$
such that for every $\epsilon > 0$, all $G \in \mathfrak{K}$, and every finite subset $S \subseteq G$,
there is a finite set $F \subseteq G$ satisfying
\begin{enumerate}
 \item $|F| \leq m(\epsilon, |S|)$ and
 \item $|SF| \leq (1+\epsilon)|F|$.
\end{enumerate}
In this case we say that $\mathfrak{K}$ is $m$-uniformly amenable.
We say that a group $G$ is uniformly amenable, if the class consisting of $G$ is uniformly amenable.

We will always assume that $m$ is non-decreasing in the second argument; this can be achieved by replacing $m$ by $m'(\epsilon,N) = \max_{k \leq N} m(\epsilon, k)$.
\end{definition}
\begin{example}
(1) Let $d \in \N$ be given. 
The class $\mathfrak{Fin}_d$ of all finite groups of order at most $d$ is uniformly amenable for the function
$m(\epsilon, N) = d$; indeed, the finite group $G$ itself is always a suitable F\o{}lner set.

(2) The class of abelian groups is uniformly amenable.

(3) Extensions of uniformly amenable groups are uniformly amenable \cite[Thm.~3]{Bozejko80}. In particular, every virtually solvable group is uniformly amenable.

(4) Direct unions of $m$-uniformly amenable groups are $m$-uniformly amenable. In particular, ascending HNN-extensions of uniformly amenable groups are uniformly amenable.

(5)
Let $\mathfrak{K}$ be a class of groups such that the free group $F_2$ of rank $2$ is residually $\mathfrak{K}$, then $\mathfrak{K}$ is not uniformly amenable (this will follow from Corollary \ref{cor:unif-amen-grps-satisfy-laws} below).
In particular, every class of groups which contains all finite symmetric groups is \emph{not} uniformly amenable.
\end{example}
\begin{example}
Let $G$ be a finite group. The direct power $G^{I}$ is uniformly amenable for every set $I$.
Given a positive integer $n$, the set of all $n$-tuples $X \defeq G^n$ in $G$ is finite. We enumerate the elements,
say $X = \{x^{(1)},\dots, x^{(k)}\}$ where $k = |G|^n$. Here $x^{(i)} = (x_1^{(i)}, \dots, x_n^{(i)})$.
Consider the group $G^{k}$ with the ``universal'' $n$-element subset $U = \{u_1, \dots, u_n\}$ where $u_j = (x_j^{(1)}, x_j^{(2)},\dots,x_j^{(k)})$.

Let $S \subseteq G^{I}$ be a subset with $n$-elements, say $S = \{s_1,\dots,s_n\}$. 
For every $i\in I$ we obtain an $n$-tuple $S(i) \defeq (s_1(i),s_2(i),\dots,s_n(i)) \in X$,
and we obtain a map $t\colon I \to \{1,\dots,k\}$ such that $S(i) = x^{(t(i))}$.
The homomorphism $\alpha\colon G^k \to G^{I}$ defined by $\alpha(g_1,\dots,g_k)(i) = g_{t(i)}$ 
maps the universal set $U$ to $S$. Therefore the subgroup generated by $S$ is isomorphic to a subfactor of $G^k$ and we deduce that $G^{I}$ is uniformly amenable.
\end{example}
\begin{definition}
A class of groups $\mathfrak{K}$ satisfies a \emph{uniform isoperimetric inequality}, if 
there is a function $\widetilde{m}\colon \R_{>0} \times \N \to \N$ such that for all $\epsilon > 0$, every $G \in \mathfrak{K}$ and every finite symmetric subset $S \subseteq G$
there is a finite subset $E \subseteq G$ with $|E| \leq \widetilde{m}(\epsilon, |S|)$ and
\[ \frac{|\partial_{S} E|}{|E|} \leq \epsilon \]
where $\partial_{S} E = SE \setminus E$ denotes the $S$-boundary of $E$.
\end{definition}
\begin{lemma}\label{lem:isoperimetric}
A class of groups $\mathfrak{K}$ is uniformly amenable if and only if it satisfies a uniform isoperimetric inequality.
\end{lemma}
\begin{proof}
Assume that $\mathfrak{K}$ is $m$-uniformly amenable. We define $\widetilde{m}(\epsilon,N) \defeq m(\epsilon, N+1)$.
Let $\epsilon > 0 $ be given, let $G \in \mathfrak{K}$ and let $S \subseteq G$ be a finite symmetric subset. We define $S^* = S \cup \{1_G\}$. Uniform amenability provides a F\o{}lner set $E \subseteq G$ with $|E| \leq m(\epsilon,N+1)$ and 
\[
	|S^*E| \leq (1+\epsilon)|E|.
\]
Since $S^*E = E \cup \partial_{S}E$ the assertion follows. 

Assume conversely that $\mathfrak{K}$ satisfies a uniform isoperimetric inequality with respect to $\widetilde{m}$. We define 
$m(\epsilon, N) = \max_{k \leq 2N} \widetilde{m}(\epsilon,k)$. Let $\epsilon > 0$, $G \in \mathfrak{K}$, and a finite set $S \subseteq G$ be given. Define $T = S \cup S^{-1}$.
By assumption, there is a finite subset $E \subseteq G$ with $|E| \leq m(\epsilon, |T|) = \max_{k \leq 2|S|} \widetilde{m}(\epsilon,k) =  m(\epsilon, |S|)$ which satisfies 
\[\frac{|\partial_T E|}{|E|} \leq \epsilon. \]
We obtain
\[ |SE| \leq |TE| \leq |E| + |\partial_T E| \leq (1+\epsilon)|E|.\qedhere\]
\end{proof}
\begin{definition}
A class of groups $\mathfrak{K}$ satisfies the \emph{uniform Reiter condition}, if there is a function $r \colon \R_{>0} \times \N \to \N$ such that for all $\epsilon > 0$, every $G \in \mathfrak{K}$, and every finite subset $S \subseteq G$, there is a finitely supported probability measure $\mu$ on $G$ such that $|\supp(\mu)| \leq r(\epsilon,|S|)$ and 
\begin{equation}\label{eq:reiter-inequ}
	\Vert \lambda^*_g(\mu)  - \mu \Vert_{\ell^1} < \epsilon
\end{equation}
for all $g \in S$. Here $\lambda^*_g(\mu)$ denotes the pullback of $\mu$ with respect to the left multiplication with $g$, i.e.,
$ \lambda^*_g(\mu) (A) = \mu(gA)$.
\end{definition}
\begin{proposition}
A class $\mathfrak{K}$  of groups is uniformly amenable if and only if it satisfies the uniform Reiter condition.
\end{proposition}
\begin{proof}
Assume that $\mathfrak{K}$ is uniformly amenable. By Lemma \ref{lem:isoperimetric} the class $\mathfrak{K}$ satisfies a uniform isoperimetric inequality w.r.t.~a function $\widetilde{m}$. Let $\epsilon >0$, $G \in \mathfrak{K}$ and $S \subseteq G$ be given. Put $S^* = S \cup S^{-1}$.
There is a finite subset $E \subseteq G$ with $|E| \leq \widetilde{m}(\epsilon, 2|S|)$ for which $\frac{|\partial_{S^*} E|}{|E|} < \epsilon$.
Let $\mu$ be the uniform probability measure supported on $E$. Since $|\partial_{S^*} E| < \epsilon |E|$, we have $|g^{-1}E \Delta E| < \epsilon |E|$ for all $g \in S$ and thus
\[
	\Vert  \lambda^*_g(\mu)  - \mu \Vert_{\ell^1} = \frac{|g^{-1}E \Delta E|}{|E|} < \epsilon
\]
Conversely, assume that $\mathfrak{K}$ satisfies the uniform Reiter condition. Let $\epsilon' > 0$, $G \in \mathfrak{K}$ and a finite symmetric subset $S \subseteq G$ be given. Set $\epsilon = \epsilon' / |S|$. Using the uniform Reiter condition, we find a finitely supported probability measure $\mu$ on $G$ with $|\supp(\mu)|\leq r(\epsilon,|S|)$ which satisfies \eqref{eq:reiter-inequ}.  For all $t \in [0,1]$ we define the level set $E_\mu(t) = \{ g \in G \mid \mu(\{g\}) \geq t\}$.
We claim that some level set satisfies a suitable isoperimetric inequality.
Summing the equality $|\lambda^*_g(\mu)(\{x\})-\mu(\{x\})| = \int_0^1 |1_{E_\mu(t)}(gx)-1_{E_\mu(t)}(x)| \;dt$ over all $x \in G$ we obtain
\[
	\Vert \lambda^*_g(\mu)-\mu\Vert_{\ell^1} = \int_0^1 \sum_{x \in G}  |1_{E_\mu(t)}(gx)-1_{E_\mu(t)}(x)| \;dt = \int_0^1 |g^{-1}E_\mu(t)\Delta E_\mu(t)| dt
\]
Taking the sum over all $g \in S$
we see that
\[
	\epsilon' = |S| \epsilon > \sum_{g \in S} \Vert \lambda^*_g(\mu)-\mu\Vert_{\ell^1} = \int_{0}^1 \sum_{g \in S} |g^{-1}E_\mu(t)  \Delta E_\mu(t)| \, dt \geq \int_0^1 |\partial_{S} E_\mu(t)| \,dt.
\]
Suppose for a contradiction that $|\partial_S E_{\mu}(t)| > \epsilon' |E_\mu(t)|$ for all $t$. Then the last integral can be estimated by 
\[
	 \int_0^1 |\partial_{S} E_\mu(t)| \,dt > \epsilon' \int_{0}^1 |E_\mu(t)| \,dt = \epsilon' 
\]
which yields a contradiction.
\end{proof}
\begin{remark}
It was proven in \cite{Wyso88} that uniform amenability of groups can be characterized by a uniform version of Kesten's condition on random walks. The argument given there -- based on a theorem of Kaimanovich (see \cite{Kaimanovich-80} or \cite[Thm.~5.2]{Kaimanovich-Vershik-83}) -- directly generalizes to classes and shows that
a uniformly amenable class of groups $\mathfrak{K}$ satisfies a \emph{uniform Kesten condition}:\\ 
There is a function $\kappa\colon \R_{>0} \times \N \to \N$ such that
for every $\epsilon  > 0$, every $G \in \mathfrak{K}$, every finitely supported symmetric probability measure $\mu$ on $G$
and all $n \geq \kappa(\epsilon, |\supp(\mu)|)$
\[
	\mathbb{P}(X_{2n} = 1_G)^{\frac{1}{2n}} > 1 - \epsilon
\]
where $X_n$ denotes the $\mu$-random walk on $G$ starting at the identity $1_G$.
\end{remark}
\begin{proposition}\label{prop:quotients}
Let $\mathfrak{K}$ be a uniformly amenable class of groups. The class of all quotients of groups in $\mathfrak{K}$ is uniformly amenable.
\end{proposition}
\begin{proof}
Assume that $\mathfrak{K}$ satisfies the uniform Reiter condition for a function $r$. Let $G \in \mathfrak{K}$, 
and let $N \subseteq G$ be a normal subgroup. The canonical projection $G \to G/N$ will be denoted by $\pi$. Given $\epsilon > 0$ and a finite subset $S \subseteq G/N$, we lift $S$ to a finite subset $S'$ in $G$, i.e., $\pi(S') = S$ and $|S'| = |S|$.
There is a finitely supported probability measure $\mu'$ on $G$ with $|\supp(\mu')| \leq r(\epsilon,|S|)$ that satisfies
\[
	\Vert \lambda^*_g(\mu') - \mu' \Vert_{\ell^1} < \epsilon
\]
for all $g \in S'$. Let $\mu = \pi_*(\mu')$ be the pushforward measure on $G/N$. Clearly the support of $\mu$ has at most as many elements as the support of $\mu'$. 
Moreover,
\begin{align*}
	\Vert \lambda^*_{gN}(\mu) - \mu\Vert_{\ell^1} & = \sum_{x \in G/N} |\mu(\{gx\})-\mu(\{x\})|
	= \sum_{x \in G/N} |\sum_{w \in x} \mu'(\{gw\})-\mu'(\{w\})| \\
	&\leq \sum_{h\in G} |\mu'(\{gh\})-\mu'(\{h\})| 
	= \Vert \lambda^*_g(\mu') - \mu' \Vert_{\ell^1} < \epsilon
\end{align*}
for all $g \in S'$. Therefore the class of factor groups satisfies the uniform Reiter condition for the same function $r$.
\end{proof}
\begin{theorem}\label{thm:residual-argument}
Let $G$ be a group and let $\mathcal{F}$ be a filter base\footnote{$\mathcal{F}$ is a non-empty set of normal subgroups, such that for all $N,M \in \mathcal{F}$ the intersection $N\cap M$ contains an element of $\mathcal{F}$. } of normal subgroups with 
$\bigcap \mathcal{F} = \{1_G\}$.
The group $G$ is uniformly amenable if and only if the class $\{G/N \mid N \in \mathcal{F} \}$ is uniformly amenable. 
\end{theorem}
\begin{proof}
Suppose $G$ is uniformly amenable. It follows immediately from Proposition \ref{prop:quotients} that $\{G/N \mid N \in \mathcal{F}\}$ is uniformly amenable. 

For the converse statement assume that $\{G/N \mid N \in \mathcal{F}\}$ satisfies the uniform Reiter condition for a function $r$ (w.l.o.g.~non-decreasing in the second argument).
Let $\epsilon > 0$ and a finite subset $S \subseteq G$ be given. We define $S_N = \pi_N(S) \subseteq G/N$, where $\pi_N\colon G \to G/N$ denotes the canonical projection for all $N \in \mathcal{F}$.

Let $C_N$ denote the set of probability measures $\mu$ on $G$ with $ |\supp(\mu)| \leq r(\epsilon, |S|)$ and whose
pushforward $\mu_N \defeq (\pi_N)_*(\mu)$ to $G/N$ satisfies
\[
	\Vert \lambda_{\pi_N(g)}^*(\mu_N) - \mu_N \Vert_{\ell^1} < \epsilon
\]
for all $g \in S$.
By assumption the sets $C_N$ are non-empty and $C_N \subseteq C_M$, if $N \subseteq M$ (this follows from the proof of Proposition~\ref{prop:quotients}).  Since the supports are bounded, the sets $C_i$ are compact in the topology of pointwise convergence. By compactness there is a measure $\nu \in \bigcap_{N \in \mathcal{F}}^\infty C_N$.
Now, let $N\in \mathcal{F}$ be sufficiently small such that distinct elements in $T = S^{-1}\supp(\nu) \cup \supp(\nu)$ represent distinct cosets in $G/N$.
For all $g \in S$ we obtain
\begin{align*}
   \Vert \lambda_g^*(\nu) - \nu \Vert_{\ell^1}  &= \sum_{x \in T}|\nu(gx)-\nu(x)| =
   \sum_{xN \in TN/N}|\nu_N(gxN)-\nu_N(xN)|\\
   &= \Vert \lambda^*_{\pi_N(g)}(\nu_N) - \nu_N \Vert_{\ell^1} < \epsilon. \qedhere
\end{align*}
\end{proof}
The following result shows that uniform amenability is a profinite property
and immediately implies Theorem \ref{thm:uniform-amenability-profinite}.
\begin{corollary}
Let $H$ be a profinite group. If some dense subgroup $G \subseteq H$ is uniformly amenable, then $H$ is uniformly amenable.
\end{corollary}
\begin{proof}
Let $\mathcal{F}$ be the filter of open normal subgroups of $H$.
Now since $G$ is uniformly amenable, the class $\{G/(G \cap N) \mid N \in \mathcal{F}\}$ is uniformly amenable. 
Since $G$ is dense in $H$, we have $G/(G\cap N) \cong H/N$ for all open normal subgroups $N \leq H$.
The uniform amenability of $H$ follows from Theorem~\ref{thm:residual-argument}.
\end{proof}
The next result is due to Keller \cite[Cor.~5.9]{Keller72}. As it will be used in the corollary below, we include a new short proof based on the uniform Kesten condition.
\begin{corollary}\label{cor:unif-amen-grps-satisfy-laws}
Every class $\mathfrak{K}$ of uniformly amenable groups satisfies a common group law.
\end{corollary}
\begin{proof}
It follows from the uniform Kesten condition that there is a number $N$ such that every pair of distinct elements in any group $G \in \mathfrak{K}$ satisfies some relation of length at most $N$.
Since there are finitely many such relations, we can form a nested commutator involving all such relations (possibly involving a new letter to make sure that the result is non-trivial); this nested commutator is a law in $G$.
\end{proof}
\begin{corollary}
Let $G$ be a finitely generated, uniformly amenable group. Then the profinite completion $\widehat{G}$ is positively finitely generated. 
\end{corollary}
\begin{proof}
Since $G$ satisfies a non-trivial law $u$, every subquotient satisfies the same law.
However, there is a finite group for which $u$ is not a law (e.g. some large symmetric group). In particular, this finite group is not a subquotient of $G$ and similarly is not a subquotient of $\widehat{G}$. It follows from 
\cite[Thm.\ 1.1]{Borovik-Pyber-Shalev} that $\widehat{G}$ is positively finitely generated.
\end{proof}
\begin{remark}
Keller asked whether a group which satisfies a law is amenable and whether an amenable group which satisfies a law is uniformly amenable. The answer to the first question is negative, since it is known by work of Adyan \cite{Adyan82} that free Burnside groups of large exponent are non-amenable. In fact, they are even uniformly non-amenable \cite{Osin07}.
On the other hand, Zelmanov's solution of the restricted Burnside problem implies that residually finite groups (not necessarily finitely generated) of bounded exponent are uniformly amenable. It is tempting to propose the following variation of Keller's question:
\end{remark}
\begin{question}
Is every family of finite groups which satisfies a common law uniformly amenable?
\end{question}
Let us close this section by noting that in general the class of uniformly amenable groups seems to be poorly understood and it would be fruitful to have more examples. Are there uniformly amenable groups which are not elementary amenable? Are there uniformly amenable groups with intermediate growth?

\section{Groups acting on rooted trees and the $\Omega$-construction}\label{sec:trees-and-construction}
The purpose of this section is to introduce the basic construction of groups we will frequently use. We begin by fixing some basic
terminology from the theory of groups acting on rooted trees.

\subsection{Groups acting on rooted trees}
By a \emph{rooted tree} we will always mean a tree $T$ with a distinguished vertex, called \emph{the root of $T$}, which we will denote by $\emptyset$.
An automorphism of $T$ will always be assumed to fix the root of $T$.
The group of all such automorphisms will be denoted by $\Aut(T)$.
Accordingly, an action of a group $G$ on a rooted tree $T$ is an action by graph isomorphisms that fix the root of $T$.
Let $V(T)$ denote the vertex set of $T$.
The distance of a vertex $v \in V(T)$ to the root $\emptyset$ is called the \emph{level} of $v$ and will be denoted by $\lv(v)$.
Two vertices $v,w \in V(T)$ are called \emph{adjacent} if they are connected by an edge.
In this paper we will mostly be interested in group actions on trees that arise as Cayley graphs of free monoids.
More precisely, let $X$ be a non-empty finite set which one can think of as an alphabet.
Let $X^{\ast}$ denote the free monoid generated by $X$, i.e.\ the set of (finite) words over $X$ with composition given by concatenation of words.
Let $T_{X}$ denote the Cayley graph of $X^{\ast}$ with respect to $X$.
Clearly $T_X$ is a tree and we consider $T_X$ as a rooted tree where the root is the empty word $\emptyset$.
Note that the set $X^{\ell}$ of words of length $\ell \in \N$ is precisely the set of vertices of level $\ell$ in $T_X$.
As every $\alpha \in \Aut(T_X)$ fixes the root of $T_X$, it follows that $\alpha$ preserves the level sets $X^{\ell}$.
Thus for every subgroup $G \leq \Aut(T_X)$ we have a natural homomorphism from $G$ to the symmetry group $\Sym(X^{\ell})$.
If $\ell = 1$, we write $\sigma_g \in \Sym(X)$ to denote the image of $g$ under this homomorphism.
On the other hand, every permutation $\sigma \in \Sym(X)$ gives rise to an automorphism of $T_X$ by defining $\sigma(xw) = \sigma(x)w$ for all $x \in X$ and $w \in X^{\ast}$.
To simplify notation, this automorphism will be denoted by $\sigma$ as well.
Automorphisms obtained in this way will be called \emph{rooted} (here we follow the terminology of \cite{BGS-branch}).
Another important type of automorphism is obtained by letting the direct sum $\Aut(T_X)^{X} \defeq \bigoplus \limits_{x \in X} \Aut(T_X)$ act on $T_X$ via $((g_x)_{x \in X},yw) \mapsto yg_y(w)$ for all $y \in X$ and $w \in X^{\ast}$.
Together with the rooted ones, these automorphism can be used to decompose arbitrary automorphism of $T_X$ as follows.

\begin{definition}\label{def:splitting-an-automorphism}
Let $X$ be a finite set and let $\alpha \in \Aut(T_X)$.
For each $x \in X$ we define the \emph{state} of $\alpha$ at $x$ as the unique automorphism $\alpha_x \in \Aut(T_X)$ that satisfies
\[
\alpha(xw) = \sigma_{\alpha}(x)\alpha_x(w)
\]
for every $w \in X^{\ast}$.
This gives us a decomposition $\alpha = \sigma_{\alpha} \circ (\alpha_x)_{x \in X}$ which is called the \emph{wreath decomposition} of $\alpha$.
\end{definition}

If the alphabet $X$ is clear from the context, we will often just write $(\alpha_x)$ instead of $(\alpha_x)_{x \in X}$.
Note that the wreath decomposition endows us with an isomorphism
\[
\Aut(T_X)
\rightarrow
\Sym(X) \ltimes \Aut(T_X)^{X},\ \alpha \mapsto \sigma_{\alpha} \cdot (\alpha_x).
\]

\begin{definition}\label{def:stabilizer-level-n}
Let $G \leq \Aut(T_X)$ be a subgroup and let $v$ be a vertex of $T_X$.
The subtree of $T_X$ whose vertex set is given by $vX^{\ast}$ will be denoted by $(T_X)_v$.
We write $\St_G(v)$ for the stabilizer of $v$ in $G$.
The \emph{rigid stabilizer} of $v$ in $G$, denoted by $\RiSt_G(v)$, is the subgroup of $\St_G(v)$ that consists of elements $g \in G$ that fix every vertex outside of $(T_X)_v$.
For $\ell \in \N_0$ we further define the \emph{level $\ell$ stabilizer subgroup}
\[
\St_G(\ell) \defeq \bigcap \limits_{v \in X^{\ell}} \St_G(v)
\]
and the \emph{rigid level $\ell$ stabilizer subgroup}
\[
\RiSt_G(\ell) \defeq \langle \bigcup \limits_{v \in X^{\ell}} \RiSt_G(v) \rangle
\]
in $G$.
\end{definition}

Let $G$ be a group that acts on a rooted tree $T$.
We call $G$ a \emph{branch group}, if the index of $\RiSt_{G}(\ell)$ in $G$ is finite for every $\ell \in \N$.
For a subgroup $G \leq \Aut(T_X)$, we say that $G$ is \emph{self-similar}, if for each $g \in G$ the elements $g_x$ in the wreath decomposition $g = \sigma_g \circ (g_x)_{x \in X}$ are contained in $G$.

\begin{notation}\label{def:iota-v}
Given a subgroup $G \leq \Aut(T_{X})$ and a word $v \in X^{\ast}$ of length $\ell$, we consider the embedding $\iota_v \colon G \rightarrow \Aut(T_{X})$ given by
\[
\iota_v(g)(uw) =
\begin{cases}
ug(w),& \text{if } u = v\\
uw,& \text{if } u \neq v
\end{cases}
\]
for every $g \in G$, $w \in X^{\ast}$ and $u \in X^{\ell}$.
\end{notation}

\subsection{The $\Omega$-construction}\label{subsec:construction}

Let us fix a non-empty finite set $X$ and an element $o \in X$.
Let $X^{+} \defeq X \setminus \{o\}$ and let $\mathcal{S}$ denote the space of infinite sequences $(\omega_n)_{n \in \N}$ over $X^{+}$.
We consider the \emph{left shift operator} $L \colon \mathcal{S} \rightarrow \mathcal{S}$ given by $(\omega_1,\omega_2,\omega_3,\ldots) \mapsto (\omega_2,\omega_3,\ldots)$.

\begin{definition}\label{def:omega-elements}
Given a sequence $\omega = (\omega_n) \in \mathcal{S}$, we define the homomorphism
\[
\widetilde{\ \cdot\ }^{\omega} \colon \Aut(T_X) \rightarrow \Aut(T_X),\ \alpha \mapsto \widetilde{\alpha}^{\omega} = (\alpha_x)_{x \in X},
\]
where
\[
\alpha_x =
\begin{cases}
\widetilde{\alpha}^{L(\omega)},& \text{if } x = o\\
\alpha,& \text{if } x = \omega_1\\
\id,              & \text{otherwise.}
\end{cases}
\]
If $G$ is a subgroup of $\Aut(T_X)$, we write $\widetilde{G}^{\omega}$ to denote the image of $G$ under $\omega$.
The group generated by $G$ and $\widetilde{G}^{\omega}$ will be denoted by $\Gamma_{G}^{\omega}$.
More generally, for every non-empty subset $\Omega \subseteq \mathcal{S}$, we define $\Gamma_{G}^{\Omega}$ as the subgroup of $\Aut(T_{X})$ that is generated by all groups $\Gamma_{G}^{\omega}$ with $\omega \in \Omega$.
\end{definition}

Let $G$ be a subgroup of $\Aut(T_X)$. 
Adapting a notion introduced by Segal~\cite{Segal01}, we say that $G$ has \emph{property H} if for all $x,y \in X$ the following hold:
\begin{itemize}
\item $G$ acts transitively on the first level of $T_X$, i.e.\ on $X$.
\item  for all $x\neq y$ in $X$ there exists $g \in \St_G(x)$ with $g(y)\neq y$.
\end{itemize}


\begin{lemma}\label{lem:structure-of-rist}
Let $G \leq \Aut(T_X)$ be a perfect, self-similar subgroup that satisfies property H.
For every $\omega = (\omega_n)_{n \in \N} \in \mathcal{S}$ we have $\iota_{\omega_1}(G) \subseteq \RiSt_{\Gamma_G^{\omega}}(\omega_1)$.
\end{lemma}
\begin{proof}
Since $G$ satisfies property~H, we can find some $h \in G$ with $h(\omega_1) = \omega_1$ and $h(o) \neq o$.
Let $h = \sigma_h \cdot (h_x)$ be the wreath decomposition of $h$.
Consider an arbitrary $g \in G$ and its image $\widetilde{g}^{\omega}$ in $\widetilde{G}^{\omega}$.
Recall that $\widetilde{g}^{\omega} = (g_x)$,
where $g_{o} = \widetilde{g}^{L(\omega)}$, $g_{\omega_1} = g$ and $g_x = \id$ otherwise.
By conjugating $\widetilde{g}^{\omega}$ with $h$ we obtain
\begin{align*}
h \widetilde{g}^{\omega} h^{-1}
&= \sigma_h \cdot (h_x) \circ (g_x) \circ (h_x^{-1}) \cdot \sigma_h^{-1}\\
&= \sigma_h \cdot (h_{x} g_{x} h_{x}^{-1}) \cdot \sigma_h^{-1}\\
&= (h_{\sigma_h(x)} g_{\sigma_h(x)} h_{\sigma_h(x)}^{-1}).
\end{align*}
From the self-similarity of $G$ we see that
\[
h_{\sigma_h(\omega_1)} g_{\sigma_h(\omega_1)} h_{\sigma_h(\omega_1)}^{-1}
= h_{\omega_1} g_{\omega_1} h_{\omega_1}^{-1}
= h_{\omega_1} g h_{\omega_1}^{-1}
\]
is an element of $G$.
Further we have $h_{\sigma_h(x)} g_{\sigma_h(x)} h_{\sigma_h(x)}^{-1} = \id$ for $\sigma_h(x) \in X \setminus \{o,\omega_1\}$.
In particular it follows that $h_{\sigma_h(o)} g_{\sigma_h(o)} h_{\sigma_h(o)}^{-1} = \id$.
Let $k \in G$ be a further element.
Note that the commutator $[\widetilde{k}^{\omega},h \widetilde{g}^{\omega} h^{-1}]$
takes the form
\[
[\widetilde{k}^{\omega},h \widetilde{g}^{\omega} h^{-1}]
= ([k_x,(h_{\sigma_h(x)} g_{\sigma_h(x)} h_{\sigma_h(x)}^{-1})])
\]
and that $[k_x,(h_{\sigma_h(x)} g_{\sigma_h(x)} h_{\sigma_h(x)}^{-1})] = \id$ for $x \neq \omega_1$.
On the other hand we have $[k_{\omega_1},(h_{\sigma_h({\omega_1})} g_{\sigma_h({\omega_1})} h_{\sigma_h({\omega_1})}^{-1})] = [k,h_{\omega_1} g h_{\omega_1}^{-1}]$.
Thus we see that every element of the form $\iota_{\omega_1}([k,h_{\omega_1} g h_{\omega_1}^{-1}])$ lies in $\RiSt_{\Gamma_{G}^{\omega}}(\omega_1)$.
Since $G$ is perfect and $g,k \in G$ were chosen arbitrarily, it follows that $\iota_{\omega_1}(G)$ is contained in $\RiSt_{\Gamma_{G}^{\omega}}(\omega_1)$.
\end{proof}

For every non-empty subset $\Omega \subseteq \mathcal{S}$ we consider its image
\[
L(\Omega) = \Set{L(\omega)}{\omega \in \Omega} \subseteq \mathcal{S}
\]
under the shift operator.

\begin{lemma}\label{lem:structure-of-rist}
Let $G \leq \Aut(T_X)$ be a perfect, self-similar subgroup that satisfies property~H.
For every non-empty $\Omega \subseteq \mathcal{S}$ and every $x \in X$ we have $\RiSt_{\Gamma_{G}^{\Omega}}(x)
= \iota_x(\Gamma_{G}^{L(\Omega)})$.
\end{lemma}
\begin{proof}
Let $\omega = (\omega_n) \in \Omega$.
By Lemma~\ref{lem:structure-of-rist} we have
$\iota_{\omega_1}(G)\subseteq \RiSt_{\Gamma_{G}^{\omega}}(\omega_1)$.
Since $G$ is self-similar and level-transitive, this implies
\begin{equation}\label{eq:structure-of-rist-2}
\iota_{o}(G)
\subseteq \RiSt_{\Gamma_{G}^{\omega}}(o)
\subseteq \RiSt_{\Gamma_{G}^{\Omega}}(o).
\end{equation}
We observe that $\widetilde{g}^{\omega}\iota_{\omega_1}(g)^{-1} = \iota_o(\widetilde{g}^{L(\omega)})$ and hence
Lemma~\ref{lem:structure-of-rist} implies further that
\begin{equation}\label{eq:structure-of-rist-3}
\iota_o(\widetilde{G}^{L(\omega)})
\subseteq \RiSt_{\Gamma_{G}^{\Omega}}(o).
\end{equation}
As $\omega \in \Omega$ was arbitrary,~\eqref{eq:structure-of-rist-2} together with~\eqref{eq:structure-of-rist-3} show that
\[
\iota_o(\Gamma_{G}^{L(\Omega)})
\subseteq \RiSt_{\Gamma_{G}^{\Omega}}(o)
\subseteq \St_{\Gamma_{G}^{\Omega}}(o).
\]
A further application of the level-transitivity and the self-similarity of $G$ now gives us $\iota_x(\Gamma_{G}^{L(\Omega)}) \subseteq \RiSt_{\Gamma_{G}^{\Omega}}(x)$ for every $x \in X$.
On the other hand, each $\Gamma_{G}^{\omega}$ is generated by elements of the form $g = \sigma_g \cdot (g_x)$ with either $g_x \in \widetilde{G}^{L(\omega)}$ or $g_x \in G$.
From this we see that the reverse inclusion $\RiSt_{\Gamma_{G}^{\Omega}}(x) \subseteq \iota_x(\Gamma_{G}^{L(\Omega)})$ is also satisfied.
\end{proof}

\begin{corollary}\label{cor:structure-of-rist-higher-level}
Let $G \leq \Aut(T_X)$ be a perfect, self-similar subgroup that satisfies property~H.
For every non-empty $\Omega \subseteq \mathcal{S}$ and every word $v \in X^{\ast}$ of length $\ell$, the restricted stabilizer of $v$ in $\Gamma_{G}^{\Omega}$ is given by $\RiSt_{\Gamma_{G}^{\Omega}}(v) = \iota_v(\Gamma_{G}^{L^{\ell}(\Omega)})$.
Moreover, we have $\RiSt_{\Gamma_{G}^{\Omega}}(\ell) = \St_{\Gamma_{G}^{\Omega}}(\ell)$ for every $\ell \in \N_0$. In particular, $\Gamma_G^{\Omega}$ is a branch group and the action is level-transitive.
\end{corollary}
\begin{proof}
The proof is by induction on the length of $v$.
If $v$ is the empty word, then there is nothing to show.
Suppose now that the corollary holds for some $\ell \in \N_0$.
Let $w \in X^{\ell+1}$ be a word of the form $w = vx$ with $v \in X^{\ell}$ and $x \in X$.
From Lemma~\ref{lem:structure-of-rist} we know that $\RiSt_{\Gamma_{G}^{L^{\ell}(\Omega)}}(x)
= \iota_x(\Gamma_{G}^{L^{\ell+1}(\Omega)})$ for every $x \in X$.
We obtain
\begin{align*}
\RiSt_{\Gamma_{G}^{\Omega}}(w)
&=&
& \RiSt_{\RiSt_{\Gamma_{G}^{\Omega}}(v)}(vx) &
&=&
&\RiSt_{\iota_v(\Gamma_{G}^{L^{\ell}(\Omega)})}(vx)&\\
&=&
&\iota_v(\RiSt_{\Gamma_{G}^{L^{\ell}(\Omega)}}(x))&
&=&
&\iota_v(\iota_x(\Gamma_{G}^{L^{\ell+1}(\Omega)}))&
&=&
&\iota_{w}(\Gamma_{G}^{L^{\ell+1}(\Omega)}).&
\end{align*}
Since $\Gamma_{G}^{\Omega}$ is generated by elements of the form $g = \sigma_g \cdot (g_x)$ with either $g_x \in G$ or $g_x \in \widetilde{G}^{L(\omega)}$ for some $\omega \in \Omega$, it follows that $\St_{\Gamma_{G}^{\omega
}}(\ell)$ is contained in the group generated by all subgroups of the form $\iota_{v}(\Gamma_{G}^{L^{\ell}(\Omega)})$ with $v$ of level $\ell$.
Together with the first part this implies $\St_{\Gamma_{G}^{\Omega}}(\ell)
= \RiSt_{\Gamma_{G}^{\Omega}}(\ell)$ and since $\St_{\Gamma_{G}^{\Omega}}(\ell)$ is of finite index, we conclude that $\Gamma_{G}^{\Omega}$ is a branch group.
By property~H the group $G$ acts transitively on the first level. Since $\RiSt_{\Gamma_G^{\Omega}}(v)$ contains $\iota_v(G)$, it follows by induction that $\Gamma_G^{\Omega}$ acts transitively on every level.
\end{proof}

To finish this section we show that the groups $\Gamma_{G}^{\Omega}$ act like iterated wreath products on each level. Recall that for group
$G,H$ with actions on sets $X$ and $Y$,
the \emph{permutational wreath product} $G \wr_X H$ is defined as the semidirect product $G \ltimes H^X$ where $G$ acts on $H^X$ by permuting the coordinates.
We define the \emph{natural action} of $G \wr_X H$ on the product set $X \times Y$ by $(g \cdot (h_x),(x,y)) \mapsto (g(x),h_x(y))$.

Given a finite set $X$ and a subgroup $Q \leq \Sym(X)$, we consider the iterated permutational wreath product of $Q$ given by
\[
\wr_X^{n} Q
= Q \wr_X (Q \wr_X (\cdots (Q \wr_X Q))\cdots)
\]
Note that the natural action of an element $\alpha \in \wr_X^{n} Q$ on $X^n$ extends to a tree automorphism on $T_X$ by setting $\alpha(vw)=\alpha(v)w$ for all $v \in X^n$ and $w \in X^{\ast}$.
In the following, we will identify $\wr_X^{n} Q$ with its image in $\Aut(T_X)$ under this action.

\begin{proposition}\label{prop:equal-image}
Let $G \leq \Aut(T_X)$ be a perfect, self-similar subgroup that satisfies property~H.
Let $Q \leq \Sym(X)$ denote the image of $G$ under the canonical action on $X$.
Then for every non-empty subset $\Omega \subseteq \mathcal{S}$ and every $\ell \in \N$, the image of $\Gamma_{G}^{\Omega}$ in $\Aut(T_{X}) / \St_{\Aut(T_{X})}(\ell)$ is given by the permutational wreath product $\wr_{X}^{\ell} Q$.
\end{proposition}
\begin{proof}
By construction, every $g \in \Gamma_{G}^{\Omega}$ has a wreath decomposition
\begin{equation}\label{eq:equal-image-1}
g = \sigma_g \circ (g_x),
\end{equation}
where $\sigma_g$ is a rooted automorphism that corresponds to an element in $Q$ and $g_x \in \Gamma_{G}^{L(\Omega)}$ for every $x \in X$.
Note that this implies that for every $g \in \Gamma_{X}^{\Omega}$ and every word $v = x_1 \dots x_{\ell}$ over $X$, its image under $g$ is given by
\begin{equation}\label{eq:equal-image-2}
g(v) = \sigma_1(x_1) \ldots \sigma_{\ell}(x_{\ell})
\end{equation}
for some appropriate permutations $\sigma_i \in Q$, i.e., the image of $\Gamma_{G}^{\Omega}$ lies in $\wr_{X}^{\ell} Q$.

On the other hand, Lemma~\ref{lem:structure-of-rist} tells us that $\RiSt_{\Gamma_{G}^{\Omega}}(x)
= \iota_x(\Gamma_{G}^{L(\Omega)})$ for every $x \in X$.
Together with~\eqref{eq:equal-image-1}, this shows that for every $\sigma \in Q$ its corresponding rooted automorphism, also denoted by $\sigma$, lies in $\Gamma_{G}^{L(\Omega)}$.
From Corollary~\ref{cor:structure-of-rist-higher-level} it therefore follows that $\iota_v(\sigma) \in \RiSt_{\Gamma_{G}^{\Omega}}(v)$ for every $v \in X^{\ast}$ and every $\sigma \in Q$.
In view of~\eqref{eq:equal-image-2}, this implies that the image of $\Gamma_{G}^{\Omega}$ in $\Aut(T_{X}) / \St_{\Aut(T_{X})}(\ell)$
is given by $\wr_{X}^{\ell} Q$.
\end{proof}
\section{The congruence subgroup property}\label{sec:csp}

Let $T$ be a rooted tree. We make $\Aut(T)$ into a topological group by declaring the subgroups $\St_{\Aut(T)}(n)$ to be a base of open neighbourhoods of the identity. Equipped with this topology the automorphism group $\Aut(T)$ is a compact, totally disconnected Hausdorff topological group, i.e., a profinite group. 

Recall that the \emph{profinite completion} of a residually finite group $G$ is defined as the inverse limit $\widehat{G} \defeq \varprojlim \limits_{N \unlhd_f G} G / N$ of the system of all normal subgroups of finite index in $G$.
If $G$ is a subgroup of $\Aut(T)$, we can further consider its \emph{tree completion} $\overline{G}$:  the closure of $G$ in $\Aut(T)$ with respect to the profinite topology. In particular $\overline{G}$ is a profinite group and $\overline{G} \cong \varprojlim \limits_{n} G / \St_G(n)$.
In this case the universal property of the profinite completion gives rise to a canonical homomorphism
 \[
\res^G_T \colon \widehat{G} \rightarrow \overline{G}.
\]
The homomorphism $\res^G_T$ allows us to extend the action of $G$ on $T$ to an action of $\widehat{G}$ on $T$.
Since $G$ is dense in both $\widehat{G}$ and $\overline{G}$, the map $\res_T$ is always surjective.
The goal of this section is to formulate sufficient conditions under which $\res_T$ is injective.

\begin{definition}\label{def:CSP}
Let $T$ be a rooted tree.
A subgroup $G \leq \Aut(T)$ satisfies the \emph{congruence subgroup property} $(\CSP)$ if $\res^G_T \colon \widehat{G} \rightarrow \overline{G}$ is an isomorphism.
\end{definition}

\begin{remark}\label{rem:rewriting-CSP}
From the definitions it directly follows that a subgroup $G \leq \Aut(T)$ satisfies the congruence subgroup property if and only if for every normal subgroup $N \unlhd G$ of finite index there is a number $n \in \N$ such that $\St_G(n)$ is contained in $N$.
\end{remark}

The following very useful observation was extracted by Segal~\cite[Lemma $4$]{Segal01} from the proof of~\cite[Theorem $4$]{Grigorchuk00}.

\begin{lemma}\label{lem:commutator-rist}
Let $T$ be a rooted tree and let $G \leq \Aut(T)$ be a subgroup that acts level transitively on $T$.
Then for every non-trivial normal subgroup $N \unlhd G$ there is some $n \in \N$ with $\RiSt_{G}(n)' \leq N$, where $\RiSt_{G}(n)'$ denotes the commutator subgroup of $\RiSt_{G}(n)$.
\end{lemma}

Recall that an infinite group $G$ is called \emph{just infinite} if every proper quotient of $G$ is finite.

\begin{corollary}\label{cor:criterion-CSP}
Let $T$ be a rooted tree and let $G \leq \Aut(T)$ be a subgroup that acts level transitively on $T$.
Suppose that every rigid stabilizer $\RiSt_G(v)$ is perfect and that the groups $\St_G(n)$ and $\RiSt_G(n)$ coincide for every $n \in \N$.
Then $G$ is just infinite and satisfies the $\CSP$.
\end{corollary}
\begin{proof}
Let $N$ be a non-trivial normal subgroup of $G$.
From Lemma~\ref{lem:commutator-rist} we know that there is some $n$ with $\RiSt_{G}(n)' \leq N$.
Since the rigid stabilizers are perfect, it follows that
\[
\RiSt_{G}(v) = \RiSt_{G}(v)' \leq \RiSt_{G}(n)'
\]
for every vertex $v$ of level $n$ in $T$.
On the other hand $\St_{G}(n) = \RiSt_{G}(n)$ is generated by the level $n$ rigid vertex stabilizers $\RiSt_{G}(v)$.
Thus we obtain $\St_{G}(n) = \RiSt_{G}(n)' \leq N$, which proves the claim.
\end{proof}

This result can be applied to the groups $\Gamma_{G}^{\Omega}$ defined in the previous section.

\begin{theorem}\label{thm:CSP-for-Omega}
Let $G \leq \Aut(T_X)$ be a perfect, self-similar subgroup that satisfies property~H.
Then for every non-empty subset $\Omega \subseteq \mathcal{S}$ the group $\Gamma_{G}^{\Omega}$ is just infinite and satisfies the congruence subgroup property.
\end{theorem}
\begin{proof}
From Corollary~\ref{cor:structure-of-rist-higher-level} we know that the groups $\RiSt_{\Gamma_{G}^{\Omega}}(\ell)$ and $\St_{\Gamma_{G}^{\Omega}}(\ell)$ coincide for every $\ell \in \N$.
As each rigid stabilizer $\RiSt_{\Gamma_{G}^{\Omega}}(v)$ is generated by isomorphic copies of the perfect group $G$, it follows that $\RiSt_{\Gamma_{G}^{\Omega}}(v)$ is perfect itself.
Now the claim follows from Corollary~\ref{cor:criterion-CSP}.
\end{proof}

As a consequence of Theorem~\ref{thm:CSP-for-Omega}, we see that the action of $\widehat{\Gamma_{G}^{\Omega}}$ on $T_{X}$ is a faithful extension of the action of $\Gamma_{G}^{\Omega}$ on $T_{X}$ and that $\widehat{\Gamma_{G}^{\Omega}}$ is isomorphic to $\overline{\Gamma_{G}^{\Omega}} \leq \Aut(T_X)$.
In the following, it will be important for us to observe that under the assumptions of Theorem \ref{thm:CSP-for-Omega} the tree completion $\overline{\Gamma_{G}^{\Omega}}$ does not depend on $\Omega$.

In fact, the tree completion is always an \emph{iterated wreath product}. Let $X$ be a finite set and let $Q \leq \Sym(X)$.
Consider the inverse limit $\wr_{X}^{\infty} Q \defeq \varprojlim \limits_n \wr_X^{n} Q$ of the iterated wreath products, where the projection $\wr_X^{n+1} Q \rightarrow \wr_X^{n} Q$ is given by restricting the natural action of $\wr_X^{n} Q$ on $X^{n+1}$ to the first $n$ coordinates.
Then the iterated wreath product $\wr_{X}^{\infty} Q$ acts on $T_X$ and we identify $\wr_{X}^{\infty} Q$ with its image in $\Aut(T_X)$ under this action.
We note that this is a closed subgroup of $\Aut(T_X)$.

Since a closed subgroup of $\Aut(T_X)$ is uniquely determined by its actions on all finite levels of the tree, the following result is a direct consequence of 
Theorem \ref{thm:CSP-for-Omega} and Proposition~\ref{prop:equal-image}.

\begin{corollary}\label{cor:profinite-completion-is-iterated-wreath}
Let $G \leq \Aut(T_X)$ be a perfect, self-similar subgroup that satisfies~H.
Let $Q \leq \Sym(X)$ denote the image of $G$ under the canonical action on $X$.
For every non-empty subset $\Omega \subseteq \mathcal{S}$, the canonical map $\res^{\Gamma_{G}^{\Omega}}_{T_{X}}$ defines an isomorphism from $\widehat{\Gamma_{G}^{\Omega}}$ onto $\wr_X^{\infty} Q \leq \Aut(T_X)$.
\end{corollary}

\begin{corollary}\label{cor:induced-isomorphism-general}
Let $G,H \leq \Aut(T_X)$ be perfect, self-similar subgroups that satisfy~H.
If the images of $G$ and $H$ in $\Sym(X)$ coincide, then the profinite completions 
$\widehat{\Gamma_{G}^{\Omega}}$ and $\widehat{\Gamma_{H}^{\Omega'}}$ are isomorphic for all non-empty subsets $\Omega,\Omega' \subseteq \mathcal{S}$.
If moreover $G$ is a subgroup of $H$ and $\Omega \subseteq \Omega'$, then $\Gamma_{G}^{\Omega}$ is a subgroup of $\Gamma_{H}^{\Omega'}$ and the inclusion map $j$ induces an isomorphism $\widehat{j} \colon \widehat{\Gamma_{G}^{\Omega}} \rightarrow \widehat{\Gamma_{H}^{\Omega'}}$.
\end{corollary}
\begin{proof}
The first assertion follows immediately from Corollary \ref{cor:profinite-completion-is-iterated-wreath}. Assume that $G \leq H$ and $\Omega \subseteq \Omega'$. By definition $\Gamma_{G}^{\Omega} \subseteq \Gamma_{H}^{\Omega'}$.
We observe that the following diagramm commutes
\[
\begin{tikzcd}
 \Gamma_{G}^{\Omega}\arrow[d]\arrow[r, "i"] & \Gamma_{H}^{\Omega'} \arrow[r]\arrow[d] & \Aut(T_X) \arrow[d,equal]\\
 \widehat{\Gamma_{G}^{\Omega}}\arrow[r, "\widehat{i}"]& \widehat{\Gamma_{H} ^{\Omega'}}\arrow[r,"\res^{\Gamma_{H} ^{\Omega'}}_{T_X}"] & \Aut(T_X)
 \end{tikzcd}
\]
and we deduce that $\res^{\Gamma_{G}^{\Omega}}_{T_X} = \res^{\Gamma_{H} ^{\Omega'}}_{T_X} \circ \widehat{i}$.
Now it follows from Corollary~\ref{cor:profinite-completion-is-iterated-wreath} that $\widehat{i}$ is an isomorphism.
\end{proof}
\section{Uncountably many groups up to isomorphism}\label{sec:uncountable}
The aim of this section is to prove that  -- under mild assumptions on $G$ -- the family of groups $\Gamma^{\Omega}_G$ where $\Omega$ runs through the non empty subsets $\Omega \subseteq \mathcal{S}$ contains uncountably many isomorphism types of groups. 

Let $G$ be a group that acts via two homomorphisms $\varphi_1,\varphi_2 \colon G \rightarrow \Aut(T)$ on a rooted tree $T$.
We say that the actions are \emph{conjugated}, if there is an automorphism $\gamma \in \Aut(T)$ such that $\varphi_2(g) = \gamma \varphi_1(g) \gamma^{-1}$ for every $g \in G$.

\begin{definition}\label{def:rigid-presentation-alternative}
Let $T$ be a tree.
We say that a subgroup $G \leq \Aut(T)$ is \emph{rigid}, if every automorphism $\alpha$ of $G$ is induced by a conjugation of $T$.
More precisely, this means that there is some $\gamma \in \Aut(T)$ with $\alpha(g) = \gamma g \gamma^{-1}$ for every $g \in G$.
\end{definition}

The following result is a special case of~\cite[Proposition 8.1]{LavrenyukNekrashevych02}.

\begin{proposition}\label{prop:rigidity-criterion}
Let $T$ be a rooted tree and let $G \leq \Aut(T)$ be a branch group.
Suppose that for every vertex $v$ the rigid stabilizer $\RiSt_G(v)$ acts level-transitively on the subtree $T_v$.
Then $G$ is rigid in $\Aut(T)$.
\end{proposition}

Recall that we write $\widehat{G}$ to denote the profinite completion of a residually finite group $G$.

\begin{lemma}\label{lem:rigidity}
Let $T$ be a rooted tree and let $G_1,G_2 \leq \Aut(T)$ be two branch groups whose restricted stabilizers $\RiSt_{G_i}(v)$ act level-transitively on $T_v$ for every vertex $v$.
Suppose that $G_1$ and $G_2$ satisfy the congruence subgroup property and that $\overline{G_1} = \overline{G_2} \subseteq \Aut(T)$.
Then every isomorphism between $G_1$ and $G_2$ is induced by a conjugation in $\Aut(T)$.
\end{lemma}
\begin{proof}
Define $\overline{G} \defeq \overline{G_1} = \overline{G_2}$.
Suppose that $f \colon G_1 \rightarrow G_2$ is an isomorphism and let $\widehat{f} \colon \widehat{G_1} \rightarrow \widehat{G_2}$ be the corresponding isomorphism on the profinite completions.
By the congruence subgroup property, the restrictions $\res^{G_i}_T \colon \widehat{G_i} \rightarrow \overline{G_i}$ are isomorphisms between the profinite completions and the tree completions. The homomorphism
 $f_0 \defeq  \res^{G_2}_{T} \circ \widehat{f} \circ (\res^{G_1}_{T})^{-1}$ is thus an automorphism of $\overline{G}$, i.e.,
the following diagram commutes
\[
\begin{tikzcd}
 \widehat{G_1}\arrow[r, "\widehat{f}"]\arrow[d,"\res^{G_1}_T",swap] & \widehat{G_2} \arrow[d, "\res^{G_2}_T"]\\
 \overline{G}\arrow[r, "f_0"]& \overline{G}
 \end{tikzcd}.
\]
Since the rigid stabilizers of $\overline{G}$ contain those of $G_1$ (and $G_2$) we can therefore apply Proposition~\ref{prop:rigidity-criterion} to deduce that there is some $\gamma \in \Aut(T)$ with $f_0(g) = \gamma g \gamma^{-1}$ for all $g \in \overline{G}$.
For every $g \in G_1 \subseteq \overline{G}$, we therefore obtain $f(g) = f_0(g) = \gamma g \gamma^{-1}$.
\end{proof}
\begin{definition}\label{def:volume-of-tree-auto}
Let $X$ be a finite alphabet and let $T_X$ be the corresponding $\abs{X}$-regular rooted tree with vertex set $X^{\ast}$.
Given a tree automorphism $g \in \Aut(T_X)$ and a number $\ell \in \N$, we consider the subset $\Fix_{\ell}(g) \subseteq X^n$ of vertices of level  $\ell$ that are fixed by $g$.
The \emph{support volume} of $g$ is defined as $\vol(g) \defeq \lim \limits_{\ell \rightarrow \infty} \frac{\abs{X^{\ell} \setminus \Fix_{\ell}(g)}}{\abs{X^{\ell}}}$.
\end{definition}
Given a tree automorphism $g \in \Aut(T)$ and a vertex $v$ with $g(v) \neq v$, it follows that no descendant of $v$ is fixed by $g$.
Thus $\frac{\abs{X^{\ell} \setminus \Fix_{\ell}(g)}}{\abs{X^{\ell}}}$ is a non-decreasing sequence of numbers that are bounded above by $1$.
In particular this tells us that the limit $\vol(g) = \lim \limits_{\ell \rightarrow \infty} \frac{\abs{X^{\ell} \setminus \Fix_{\ell}(g)}}{\abs{X^{\ell}}}$ indeed exists. In fact, the support volume measures the set of elements in the boundary of $T_X$ which are moved by $g$.
The support volume is invariant under conjugation. Let $\alpha \in \Aut(T)$ be an automorphism. Then $\Fix_{\ell}(\alpha g \alpha^{-1}) = \alpha(\Fix_\ell(g))$
and hence $\vol(g) = \vol(\alpha g \alpha^{-1})$.

We return to the construction introduced in Section~\ref{subsec:construction}. Let $X$ be a non-empty finite set with an element $o \in X$ and define $X^+ = X \setminus \{o\}$.  Recall that $\mathcal{S} \defeq (X^+)^\infty$.

\begin{theorem}\label{thm:continuum-of-volume-simple}
Let $X$ be a finite set with $|X|\geq 3$.
Let $G \leq \Aut(T_X)$ be a non-trivial subgroup. 
For every $\omega \in \mathcal{S}$ the set of real numbers
\[
\Set{\vol(g)}{g \in \Gamma^{\{\omega,\omega'\}}_G, \omega' \in \mathcal{S}} \subseteq [0,1]
\]
is uncountable.
\end{theorem}
\begin{proof}
Let $t \in G$ be an element, which acts non-trivially on $T_X$; in particular, $\vol(t) > 0$.
Since $|X| \geq 3$, we can pick an element $z_n \in X \setminus \{o,\omega_n\}$ for every $n \in \N$.
Let $S \subseteq \N$ be a set of natural numbers. 
We define $\omega' = \omega'(S)$ by
\[
	\omega_n' =  \begin{cases} \omega_n & \text{ if } n \not\in S\\
								z_n & \text{ if } n \in S. \end{cases}
\]
Consider the element 
$g = (\widetilde{t}^{\omega})^{-1}\widetilde{t}^{\omega'} \in \Gamma_G^{\{\omega,\omega'\}}$. Then $g$ acts like $t^{-1}$ on $(T_X)_{o^{n-1}\omega_n}$
and like $t$ on $(T_X)_{o^{n-1}z_n}$ for every $n \in S$ and acts trivially on all vertices not contained in one of these subtrees. We obtain
\[
	\vol(g) = \sum_{n \in S} \frac{2\vol(t)}{|X|^{n}} = 2\vol(t) \sum_{n \in S} |X|^{-n}
\]
and we observe that this number uniquely determines the set $S$.
Indeed, since $|X|\geq 3$ the first non-zero term dominates the series. This completes the proof of the theorem, using that there are uncountable many subsets $S \subseteq \N$.
\end{proof}
\begin{corollary}\label{cor:uncountably-many-iso-types-simple}
Let $G \leq \Aut(T_X)$ be a countable, perfect, self-similar subgroup that satisfies property~H.
For every $\omega \in \mathcal{S}$ 
the family of groups $(\Gamma_G^{\{\omega,\omega'\}})_{\omega' \in \mathcal{S}}$ contains uncountably many distinct isomorphism types.
\end{corollary}
\begin{proof}
Recall that
by Corollary \ref{cor:structure-of-rist-higher-level} the groups $\Gamma_G^{\Omega}$ are branch groups and the rigid stabilizers act level-transitively. By Theorem \ref{thm:CSP-for-Omega} these groups have the congruence subgroup property and by Corollary~\ref{cor:profinite-completion-is-iterated-wreath} the closure of $\Gamma_G^{\Omega}$ in $\Aut(T_X)$ does not depend on $\Omega$. We conclude using Lemma \ref{lem:rigidity} that every isomorphism between two of the groups $\Gamma_G^{\Omega}$ is induced by a conjugation in $\Aut(T_X)$. In particular, this means, that isomorphisms between these groups preserve the support volume of elements.

We note that $G$ is perfect and acts transitively on $X$, hence we must have $|X| \geq 5$. Theorem \ref{thm:continuum-of-volume-simple} therefore shows, that the set of support volumes of elements in the groups $\Gamma_G^{\{\omega,\omega'\}}$ is uncountable. However, $G$ is countable and so the groups $\Gamma_G^{\{\omega,\omega'\}}$ are countably generated and thus countable.
In conclusion each isomorphism type contributes at most countably many numbers to the uncountable set of support volumes and consequently uncountably many isomorphisms types have to occur. 
\end{proof}
In the next section we will discuss a concrete example of a group $G$ where a similar argument can be used to show that the numer of isomorphism types in the family $(\Gamma_G^\omega)_{\omega \in \mathcal{S}}$ is uncountable. 
\begin{remark}
We briefly return to Grothendieck's question.
If $G \leq \Aut(T)$ is a finitely generated group which satisfies the assumptions of Corollary \ref{cor:uncountably-many-iso-types-simple}, then the groups $(\Gamma_{G}^\Omega)_{\Omega \subseteq \mathcal{S}}$ where $\Omega$ runs in the finite subsets of $\mathcal{S}$ form an uncountable directed system of finitely generated residually finite groups in which every inclusion induces an isomorphism between profinite completions (see Corollary~\ref{cor:induced-isomorphism-general}).
\end{remark}

\section{Matrix groups acting on trees}\label{sec:matrix-groups}

Given a prime number $p$ and a natural number $n$ we consider the
set $\mathcal{A}_{p,n} \defeq \{0,\ldots,p-1\}^n$ which takes the role of the alphabet (called $X$ in the previous sections).
Let $\mathcal{A}_{p,n}^{\ast}$ denote the free monoid generated by $\mathcal{A}_{p,n}$, i.e.\ the set of (finite) words over $\mathcal{A}_{p,n}$.
Let $T_{p,n}$ denote the Cayley graph of $\mathcal{A}_{p,n}^{\ast}$ with respect to $\mathcal{A}_{p,n}$.
Clearly $T_{p,n}$ is a tree whose boundary $\partial T_{p,n}$ can be identified with the set $\mathcal{A}_{p,n}^{\infty}$ of infinite sequences over $\mathcal{A}_{p,n}$.
The element $0 = (0,0,\dots,0) \in \mathcal{A}_{p,n}$ is the distinguished element and
we write $\mathcal{S}_{p,n}$ to denote the space of infinite sequences over $\mathcal{A}_{p,n} \setminus \{0\}$.

\begin{definition}\label{def:affine-group}
Given a commutative, unital ring $R$ and a natural number $n \in \N$ we write $\SAff_n(R)$ to denote the group of affine transformations of $R^n$ whose linear part lies in $\SL_n(R)$. We note that $\SAff_n(R) \cong R^n \rtimes \SL_n(R)$.
\end{definition}

It is a well-known fact that $\SL_n(\Z)$ and $\SL_n(\F_p)$ are perfect for $n \geq 3$ (see for example~\cite[1.2.15]{Hahn-OMeara} and \cite[p.~46]{Wilson-FSG}).
In the following we need an affine version of this observation.

\begin{lemma}\label{lem:Affn-is-perfect}
The groups $\SAff_n(\Z)$ and $\SAff_n(\F_p)$ are perfect for $n \geq 3$.
\end{lemma}
\begin{proof}
For every $v \in \Z^n$ let $T_v$ denote the translation by $v$.
Since $\SL_n(\Z)$ is perfect (for $n \geq 3$), it suffices to show that every translation $T_{e_i}$ by a standard unit vector $e_i \in \Z^n$ can be written as a commutator in $\SAff_n(\Z)$.
To see this, we consider the elementary matrices $E_{i,j} \in \SL_n(\Z)$ for $1 \leq i < j \leq n$ which are defined by $E_{i,j} \cdot e_j = e_i+e_j$ and $E_{i,j} \cdot e_k = e_k$ for $k \neq j$.
Then
\[
[T_{-e_j},E_{i,j}]
= T_{-e_j} E_{i,j} T_{-e_j}^{-1} E_{i,j}^{-1}
= T_{-e_j} + T_{E_{i,j} \cdot e_j}
= T_{-e_j} + T_{e_i+e_j}
= T_{e_i}
\]
and the result for $\SAff(\Z)$ follows. The same argument applies to $\SAff(\F_p)$.
\end{proof}

The set $\mathcal{A}_{p,n}^{\infty}$ can be identified with $\Z_p^n$ via $(x_n) \mapsto \sum \limits_{n=0}^{\infty} p^i x_i$
and similarly the $\ell$-th level of the tree can be identified with $(\Z/p^\ell\Z)^n$.
In view of this, the natural action of $\SAff_n(\Z_p)$ on $\Z_p^n$ induces an action on $T_{p,n}$. In fact, the action on the $\ell$-th level  factors through $\SAff_n(\Z / p^{\ell}\Z)$.

\begin{lemma}\label{lem:affine-stabilizers}
The subgroup $\SAff_n(\Z) \leq \Aut(T_{p,n})$ is self-similar and satisfies property~H.
\end{lemma}
\begin{proof}
Let $A \in \SL_n(\Z)$, let $b \in \Z^n$, and let $g \in \SAff_n(\Z)$ be the element defined by $g(v) = Av + b$.
Let $u \in \Z_p$ be an element of the form $u = x + pw$ with $x \in \mathcal{A}_{p,n}$ und $w \in \Z_p$.
Let further $x' \in \mathcal{A}_{p,n}$ and $b' \in \Z_p$ be such that $Ax+b = x' + pb'$.
Then we have
\[
g(u)
= A(x + pw) + b
= Ax+b + pAw
= x' + p(Aw+b'),
\]
which tells us that $g_x$ is given by $g_x(w) = Aw+b'$.
As $Ax+b = x' + pb'$ implies $b' \in \Z$ and $g_x \in \SAff_n(\Z)$, we deduce that $\SAff(\Z)$ is self-similar.

The action of $\SAff(\Z)$ on the first level $A_{p,n}$ factors through $\SAff(\F_p)$. In fact, this is the natural action of $\SAff(\F_p)$ on $\F_p^n$. Since this action is $2$-transitive, it clearly satisfies property~H.
\end{proof}

\begin{definition}\label{def:collatz-automorphism}
Let $\omega = (\omega_n)_{n \in \N} \in \mathcal{S}_{p,n}$.
For $g \in \SAff_n(\Z)$ we define the map $\widetilde{g}^{\omega} \colon \Z_p^n \rightarrow \Z_p^n$ by
\[
\widetilde{g}^{\omega}(u) =
\begin{cases}
p^{\ell-1} \omega_{\ell} + p^{\ell}g(v),& \text{if } u = p^{\ell-1} \omega_{\ell} + p^{\ell}v \text{ for some } \ell \in \N \text{ and } v \in \Z_p^n\\
u,& \text{if } u \not\equiv p^{\ell-1} \omega_{\ell} \mod p^{\ell} \text{ for every } \ell \in \N.
\end{cases}
\]
Further we define $\widetilde{\SAff_n(\Z)}^{\omega} \defeq \Set{\widetilde{g}^{\omega}}{g \in \SAff_n(\Z)}$.
\end{definition}

From this definition one can easily see that $\widetilde{\SAff_n(\Z)}^{\omega}$ is a group and that
\[
\widetilde{\ \cdot \ }^{\omega} \colon \SAff_n(\Z) \rightarrow \widetilde{\SAff_n(\Z)}^{\omega},\ g \mapsto \widetilde{g}^{\omega}
\]
is a group isomorphism.
The elements $\widetilde{g}^{\omega}$ can also be defined recursively with the \emph{left shift operator}
\[
L \colon \mathcal{A}_{p,n}^{\infty} \rightarrow \mathcal{A}_{p,n}^{\infty},\ (x_1,x_2,x_3,\ldots) \mapsto (x_2,x_3,x_4\ldots).
\]

Indeed, given a sequence $\omega = (\omega_n)_{n \in \N} \in \mathcal{S}_{p,n}$ and an element $g \in \SAff_n(\Z)$ then we can write
\[
\widetilde{g}^{\omega} = (g_x)_{x \in \mathcal{A}_{p,n}},
\]
where
\[
g_x =
\begin{cases}
\widetilde{g}^{S(\omega)},& \text{if } x = 0\\
g,& \text{if } x = \omega_1\\
\id,              & \text{otherwise.}
\end{cases}
\]
This is exactly the formula used in Definition \ref{def:omega-elements}. For every subset $\omega \in \mathcal{S}_{p,n}$ we define the subgroup  $\Gamma_{p,n}^{\omega} \leq \Aut(T_{p,n})$ to be the group generated by $\SAff_n(\Z)$ and $\widetilde{\SAff_n(\Z)}^{\omega}$. Recall that for a set $\Omega \subseteq \mathcal{S}_{p,n}$ we define the subgroup $\Gamma_{p,n}^{\Omega} \leq \Aut(T_{p,n})$ to be generated by the groups $\Gamma_{p,n}^{\omega}$ with $\omega \in \Omega$.
Lemma \ref{lem:Affn-is-perfect} and Lemma \ref{lem:affine-stabilizers} allow us to use the results developed in the foregoing sections.
In particular, we obtain the following result.

\begin{corollary}\label{cor:summary}
Let $n \geq 3$ and let $\Omega, \Omega_1, \Omega_2 \subseteq \mathcal{S}_{p,n}$ be non-empty subsets.
\begin{enumerate}
\item Then $\Gamma_{p,n}^{\Omega}$ is a level-transitive, just infinite branch group which contains a non-abelian free group and satisfies the congruence subgroup property.
The profinite completion is isomorphic to the closure of $\Gamma_{p,n}^{\Omega}$ in $\Aut(T_{p,n})$ and does not depend $\Omega$.
\item If $\Gamma_{p,n}^{\Omega_1}$ and $\Gamma_{p,n}^{\Omega_2}$ are isomorphic, then they are already conjugated in $\Aut(T_{p,n})$.
\item For $\Omega_1 \subseteq \Omega_2$ the inclusion  $\Gamma_{p,n}^{\Omega_1} \to \Gamma_{p,n}^{\Omega_2}$ induces an isomorphism between the profinite completions.
\end{enumerate}
\end{corollary}
\begin{proof}
It is well-known that $\SL_n(\Z)$ contains non-abelian free subgroups.
By Corollary \ref{cor:structure-of-rist-higher-level} the groups $\Gamma_{p,n}^{\Omega}$ are branch groups and the rigid stabilizers act level-transitively. By Theorem \ref{thm:CSP-for-Omega} these groups have the congruence subgroup property and by Corollary~\ref{cor:profinite-completion-is-iterated-wreath} the closure of $\Gamma_{p,n}^{\Omega}$ in $\Aut(T_{p,n})$ is isomorphic to the profinite completion and does not depend on $\Omega$. Lemma \ref{lem:rigidity} shows that every isomorphism between two of the groups is induced by a conjugation in $\Aut(T_{p,n})$.
The third assertion follows from Corollary~\ref{cor:induced-isomorphism-general}.
\end{proof}
It follows from Corollary~\ref{cor:uncountably-many-iso-types-simple} that the number of isomorphism types among the groups $\Gamma_{p,n}^{\{\omega, \omega'\}}$ is uncountable. A variation of the argument shows that we can also find uncountably many groups up to isomorphism in the family $(\Gamma_{p,n}^{\omega})_{\omega \in \mathcal{S}_{p,n}}$. In particular, most of these groups don't admit a finite presentation.

\begin{proposition}\label{prop:continuum-of-isom-classes}
For every $n \geq 3$ and every prime $p$  there are uncountably many isomorphism classes of groups of the form $\Gamma_{p,n}^{\omega}$.
\end{proposition}
\begin{proof}
If two of the groups $\Gamma_{p,n}^{\omega}$ are isomorphic, then they are conjugated in $\Aut(T_{p,n})$ (see Corollary~\ref{cor:summary}) and since 
conjugation in $\Aut(T_{p,n})$ preserves support volumes, we deduce that for isomorphic groups the sets
\[ 
	\vol(\Gamma_{p,n}^\omega) = \{\vol(g) \mid g \in \Gamma_{p,n}^\omega \}
\]
coincide. We note that the groups $\Gamma_{p,n}^\omega$ are finitely generated and thus $\vol(\Gamma_{p,n}^\omega)$ is a countable set.
In particular, it is sufficent to prove -- following  Theorem \ref{thm:continuum-of-volume-simple} -- that the set
	$\bigcup_{\omega \in \mathcal{S}_{p,n}} \vol(\Gamma_{p,n}^\omega)$
is uncountable.

Let $e_1,e_2,\dots, e_n$ denote the standard basis of $\Z^n$.
Consider the elementary matrix $A = E_{1,2} \in \SL_n(\Z)$ with $Ae_1 = e_1$ and $Ae_2 = e_1 + e_2$.
Let $T = T_{e_1} \in \SAff_n(\Z)$ be the translation with the first standard basis vector.
For every subset $S \subseteq \N$ we define
$\omega = \omega(S) \in \mathcal{S}_{p,n}$ such that
\[
	\omega_i = \begin{cases} e_1 \; & \text{ if } i \not\in S\\
							e_2 & \text{ if } i \in S\end{cases}.
\]
Let $\omega' = A \omega$.
Using the formula given in Definition \ref{def:collatz-automorphism} it is readily checked that $A\widetilde{T}^{\omega}A^{-1} = \widetilde{T}^{\omega'}$.
We consider the commutator $g = [A,\widetilde{T}^{\omega}] \in \Gamma_{p,n}^{\omega}$ and we observe that
\[
	g = [A,\widetilde{T}^{\omega}] = A \widetilde{T}^{\omega} A^{-1} (\widetilde{T}^{\omega})^{-1} = \widetilde{T}^{\omega'} (\widetilde{T}^{\omega})^{-1}
\]
In particular, $g$ acts non-trivially exactly on the boundary points $x \in \Z_p^n$ congruent to
$p^{i}e_2$ or $p^{i}(e_1+e_2)$ with $i \in S$  .
We obtain
\[
	\vol(g) = \sum_{i \in S} \frac{2}{p^{ni}} = 2 \sum_{i \in S} p^{-ni}
\]
and we observe that this number uniquely determines the set $S$. Since there are uncountably many subsets $S \subseteq \N$, this completes the proof.
\end{proof}

\section{Amenable groups acting on trees}\label{sec:amenable}

Let $X$ be a finite set, let $o \in X$ and let $X^{+} \defeq X \setminus \{o\}$.
Let $\mathcal{S}$ denote the set of infinite sequences over $X^{+}$.
Our goal in this section is to introduce amenable groups that have the same profinite completions as $\Gamma_{p,n}^{\Omega}$ for $n \geq 3$.
To this end we introduce automatic automorphisms of $T_X$.
Recall that for every vertex $v \in T_X$, we write $(T_X)_v$ to denote the subtree of $T_X$ whose vertex set is given by $vX^{\ast}$.
For $\alpha \in \Aut(T_X)$ we have $\alpha((T_X)_v) = (T_X)_{\alpha(v)}$.
Thus we can define the \emph{state of $\alpha$ at $v$} as the unique automorphism $\alpha_v$ of $T_X$ that satisfies $\alpha(vw) = \alpha(v)\alpha_v(w)$ for every $w \in X^{\ast}$.
The set of all states of $\alpha$ will be denoted by $S(\alpha) \defeq \Set{\alpha_v \in \Aut(T_X)}{v \in X^{\ast}}$.

\begin{definition}\label{def:automatic-automorphism}
An automorphism $\alpha$ of $T_X$ is called \emph{automatic}, if $S(\alpha)$ is finite.
\end{definition}

\begin{example}\label{exam:automatic-automorphism}
Let $\sigma$ be a rooted automorphism of $T_X$ and let $\omega = (\omega_{\ell})_{\ell \in \N} \in \mathcal{S}$.
Consider the automorphism $\widetilde{\alpha}^{\omega}$ of $T_X$.
For $v \in X^{\ast}$ we have
\[
\widetilde{\alpha}^{\omega}_v =
\begin{cases}
\widetilde{\alpha}^{L^{\ell}(\omega)},& \text{if } v = o^{\ell} \text{ for some } \ell \in \N_0\\
\alpha,& \text{if } v = o^{\ell}\omega_{\ell} \text{ for some } \ell \in \N_0\\
\id,              & \text{otherwise.}
\end{cases}
\]
Thus the set of states of $\widetilde{\alpha}^{\omega}$ is finite if and only if $\Set{L^{\ell}(\omega) \in \mathcal{S}}{\ell \in \N}$ is a finite subset of $\mathcal{S}$.
From this we see that $\widetilde{\alpha}^{\omega}$ is automatic if and only if there is some $N \in \N$ such that $L^{N}(\omega)$ is periodic.
\end{example}

\begin{definition}\label{def:bounded-automorphism}
An automorphism $\alpha \in \Aut(T_X)$ is called \emph{bounded}, if there is some $C \geq 0$ such that
\[
\abs{\Set{v \in X^{\ell}}{\alpha_v \neq \id}} \leq C
\]
for all $\ell \in \N_0$.
\end{definition}

\begin{example}\label{exam:bounded-automorphism}
Let $\sigma$ be a rooted automorphism of $T_X$.
Then $\widetilde{\alpha}^{\omega}$ is clearly bounded for every choice of $\omega \in \mathcal{S}$.
\end{example}

It can be easily seen that the set of all bounded automatic automorphisms of $T_X$ forms a group.
In~\cite[Theorem 1.2]{BartholdiKaimanovichNekrashevych10}, Bartholdi, Kaimanovich and Nekrashevych proved that this group is amenable.
As subgroups of amenable groups are amenable, it follows that every subgroup of $\Aut(T_X)$ that is generated by bounded automatic automorphisms is amenable.
In view of Example~\ref{exam:automatic-automorphism} and Example~\ref{exam:bounded-automorphism} we therefore obtain the following.
\begin{proposition}\label{prop:amenability-criterion}
Let $G \leq \Aut(T_X)$ be a group of rooted automorphisms and let $\Omega \subseteq \mathcal{S}$ be a non-empty subset.
Suppose that every $\omega \in \Omega$ is evenually periodic, i.e.\ there is a $N_{\omega} \in \N$ such that $L^{N_{\omega}}(\omega)$ is periodic.
Then $\Gamma_{G}^{\Omega}$ is amenable.
\end{proposition}

Now we are able to proof Theorem \ref{thm:main-theorem}.
Let $p$ be a prime and let $n \geq 3$ be a natural number.
Consider the natural action of $\SAff_n(\F_p)$ on $\mathcal{A}_{p,n} \defeq \{0,\ldots,p-1\}^{n}$.
Let $\mathcal{A}_{p,n}^{+}$ denote the complement of $0 \defeq (0,0,\dots,0)$ in $\mathcal{A}_{p,n}$ and let $\mathcal{S}_{p,n}$
denote the set of sequences in $\mathcal{A}_{p,n}^+$.
For every non-empty set $\Omega \subseteq \mathcal{S}_{p,n}$, we write $A_{p,n}^{\Omega} \defeq \Gamma_{\SAff_n(\F_p)}^{\Omega}$, where $\SAff_n(\F_p)$ is identified with the corresponding group of rooted automorphisms of $T_{p,n} \defeq T_{\mathcal{A}_{p,n}}$.
Let further $G_{p,n}$ denote the subgroup of $\Aut(T_{p,n})$ that is generated by the canonical actions of $\SAff_n(\F_p)$ and $\SAff_n(\Z)$ on $T_{p,n}$.
Let $M_{p,n}^{\Omega} \defeq \Gamma_{G_{p,n}}^{\Omega}$ be the corresponding $\Omega$-group. Equivalently, $M_{p,n}$ is the subgroup of $\Aut(T_{p,n})$ that is generated by $A_{p,n}^{\Omega}$ and $\Gamma_{p,n}^{\Omega}$.

\begin{theorem}\label{thm:main-theorem-precise}
Let $n \geq 3$ be a natural number, let $p$ be a prime, let $\omega \in \mathcal{S}_{p,n}$ be eventually periodic,
and let $\Omega \subseteq \mathcal{S}_{p,n}$ be a finite subset.
Then the following hold:
\begin{enumerate}
\item $A_{p,n}^{\omega}$ is a finitely generated amenable group.
\item $M_{p,n}^{\Omega}$ is finitely generated and contains a non-abelian free group.
\item If $\omega \in \Omega$, then the inclusion $\iota \colon A_{p,n}^{\omega} \rightarrow M_{p,n}^{\Omega}$ induces an isomorphism $\widehat{\iota} \colon \widehat{A_{p,n}^{\omega}} \rightarrow \widehat{M_{p,n}^{\Omega}}$ of profinite completions.
\item The family $(M_{p,n}^{\{\omega,\omega'\}})_{\omega' \in \mathcal{S}_{p,n}}$ contains uncountably many pairwise non-isomorphic groups.
\end{enumerate}
\end{theorem}
\begin{proof}
The first assertion follows from Proposition \ref{prop:amenability-criterion}.
Since $\Omega$ is finite, it follows that $M_{p,n}^\Omega$ is finitely generated. Since $M^\Omega_{p,n}$ contains
$\Gamma_{p,n}^\Omega$, it contains a non-abelian free subgroup by \ref{cor:summary}.

To prove the third assertion, 
we verify the assumptions of Corollary~\ref{cor:induced-isomorphism-general}.
First we observe that the groups $\SAff_n(\F_p)$ and $\langle \SAff_n(\F_p) \cup \SAff_n(\Z) \rangle$ are perfect (see Lemma \ref{lem:Affn-is-perfect}). In addition these groups are self-similar and satisfy property~H; to see this one can use the argument given in Lemma \ref{lem:affine-stabilizers}.

Finally, it follows from \ref{cor:uncountably-many-iso-types-simple} that the family $(M_{p,n}^{\{\omega,\omega'\}})_{\omega' \in \mathcal{S}_{p,n}}$ of subgroups contains uncountably many pairwise non-isomorphic groups
\end{proof}
In order deduce Theorem \ref{thm:main-theorem} from Theorem \ref{thm:main-theorem-precise}, it remains to determine the number of generators. It is known that $\SL_n(\Z)$ and $\SL_n(\F_p)$ are $2$-generated (see \cite{HuaReiner49}), 
and so $\SAff_n(\Z)$ and $\SAff_n(\F_p)$ can be generated by $3$ elements.
Since $A_{p,n}^{\omega}$ is generated by two copies of $\SAff_n(\F_p)$, it is $6$-generated. Similary the group $G_{p,n}$ is $6$-generated and so $M_{p,n}^{\{\omega,\omega'\}}$ -- which is generated be three copies of $G_{p,n}$ -- can be generated using $18$ elements.

\bibliographystyle{amsplain}
\bibliography{literatur}

\end{document}